\documentclass[draftcls,onecolumn,12pt]{IEEEtran}

\usepackage{graphicx}          
\usepackage{dsfont}
\usepackage{bbm}
\usepackage{amsmath}
\usepackage{graphics} 
\usepackage{epsfig} 
\usepackage{amssymb}  
\usepackage{subfigure}
\usepackage{setspace}
\usepackage{bbold,mathtools}
\usepackage{algorithm}
\usepackage{algorithmic}

\newenvironment{proof}[1][Proof]{\begin{trivlist}
\item[\hskip \labelsep {\bfseries #1}]}{\end{trivlist}}

\newtheorem{theorem}{Theorem}
\newtheorem{remark}{Remark}
\newtheorem{lemma}{Lemma}

\renewcommand{\baselinestretch}{1.5}

\begin{document}

\title{\LARGE \bf
Minimum Cost Input/Output Design for Large-Scale Linear Structural Systems}
%
%
%
%
%

\author{\IEEEauthorblockN{  S\'ergio Pequito $^{\dagger,\ddagger}$ \qquad  Soummya Kar $^{\dagger}$ \qquad
A. Pedro Aguiar $^{\ddagger,\diamond}$ }

\thanks{ 

$^{\dagger}$ Department of Electrical and Computer Engineering, Carnegie Mellon University, Pittsburgh, PA 15213

$^{\ddagger}$  Institute for System and Robotics, Instituto Superior T\'ecnico, Technical University of Lisbon, Lisbon, Portugal

$^{\diamond}$ Department of Electrical and Computer Engineering, Faculty of Engineering, University of Porto, Porto, Portugal

}

}

\maketitle

\vspace{-1cm}

\begin{abstract}                          
In this paper, we provide optimal  solutions to two different (but related) input/output design  problems involving large-scale linear dynamical systems, where the cost associated to each directly actuated/measured state variable can take different values, but is independent of the labeled input/output variable. Under these conditions, we first aim to determine and characterize the input/output placement that incurs in the minimum cost while ensuring that the resulting placement achieves structural controllability/observability. Further, we address a constrained variant of the above problem, in which we seek to determine the minimum cost placement configuration, among all possible input/output placement configurations that ensures structural controllability/observability, with the lowest number of directly actuated/measured state variables. We show that both  problems can be solved efficiently, i.e., using  algorithms with polynomial time complexity in the number of the state variables. Finally, we illustrate the obtained results with an example.

\vspace{0.3cm}
\end{abstract}

\IEEEpeerreviewmaketitle

\section{Introduction}

The  problem of control systems  design, meeting certain desired specifications, is of fundamental importance. Possible specifications include (but are not restricted  to) controllability and observability. These specifications ensure the capability of a dynamical system (such as chemical process plants, refineries, power plants, and airplanes, to name a few) to drive its state toward a specified goal or infer its present  state. To achieve these specifications, the selection of where to place the actuators and sensors assumes a critical importance.   More often than not, we need to consider the cost per actuator/sensor, that depends on its specific funtionality and/or its installation and maintenance cost. The resulting placement cost optimization problem  (apparently combinatorial) can be quite non-trivial, and currently applied  state-of-the-art methods typically consider  relaxations of the optimization problem, brute force approaches or heuristics, see for instance~\cite{reviewPlacement,Frecker01042003,WilliamBegg200025,1025322,Fahroo2000137}.

An additional problem is the fact  that the precise numerical values of the system model parameters are generally not available for many large-scale systems of interest. A natural direction is to consider structural systems~\cite{dionSurveyKyb} based reformulations, which we pursue in this work. Representative work in structural systems theory may be found in \cite{Lin_1974,largeScale,Reinschke:1988,Murota:2009:MMS:1822520}, see also the survey \cite{dionSurvey} and references therein. The main idea is to reformulate and  study  an equivalent class of systems for which system-theoretic properties are investigated on the basis of  the location of zeros/nonzeros of the state space representation matrices. Properties such as controllability and observability are, in this framework, referred to  as  \emph{structural controllability}\footnote{A pair $(A,B)$ is said to be structurally controllable if there exists a pair $(A', B')$ with the same structure as $(A,B)$, i.e., same locations of zeros and nonzeros, such that $(A',B')$ is controllable. By density arguments \cite{Reinschke:1988}, it can be shown that if a pair $(A,B)$ is structurally controllable, then almost all (with respect to the Lebesgue measure) pairs with the same structure as $(A,B)$ are controllable. In essence, structural controllability is a property of the structure of the pair $(A,B)$ and not of the specific numerical values. A similar definition and characterization holds for structural observability (with obvious modifications). } and \emph{structural observability}, respectively.

In this context, consider a given (possibly large-scale) system with autonomous dynamics
\begin{equation}
\dot x = A x,
\label{dynSys}
\end{equation}
where $x \in\mathbb{R}^n$ denotes the state and $A$ is the  $n\times n$ dynamics matrix. Suppose that the sparsity (i.e., location of zeros and nonzeros) pattern of $A$ is available, but the specific numerical values of its nonzero elements are not known. Let $\bar A \in \{0,1\}^{n\times n}$ be the binary matrix that represents the structural pattern of $A$, i.e., it encodes the sparsity pattern of $A$  by assigning $0$ to each zero entry of $A$ and $1$ otherwise. The two problems  addressed in this paper are as follows.

\subsection*{\underline{Problem Statement}}

\begin{itemize}
\item[${\mathcal{P}}_1$]  

Given the structure of the dynamics matrix $\bar A\in\{0,1\}^{n\times n}$ and a vector $c$ of size $n$, where the entry $c_i\ge 0$ denotes the cost of directly actuating the state variable $i$, determine the sparsity of the input matrix $\bar B$  that solves the following optimization problem
\begin{align}
\min_{\bar B\in \{0,1\}^{n \times n}} &\qquad\qquad  \|\bar B\|_{c}\label{probP1}\\
\text{s.t.} \qquad &  (\bar A,\bar B) \text{ structurally controllable}\notag\\
& \quad \|\bar B\|_0\le \|\bar B'\|_0 , \text { for all }\notag\\  &\quad (\bar A,\bar B') \text{ structurally controllable},\notag
\end{align}
where $
\|\bar B\|_{c}=c^T\bar B \mathbf{1}$,  
$\|\bar B\|_0$ denotes the zero (quasi) norm corresponding to the number of nonzero entries in $\bar{B}$, and $\mathbf{1}$ the vector of ones with size $n$.
\vspace{0.3cm}

Observe that in $\mathcal P_1$ there is a restriction of obtaining a solution with the minimum number of state variables that need to be directly actuated in order to achieve structural controllability. Without such restriction, i.e., by possibly actuating more state variables, we may obtain a lower cost placement achieving structural controllability, which leads to the following related problem.

\vspace{0.3cm}

\item[${\mathcal{P}}_2$]  

 Given the structure of the dynamics matrix $\bar A\in\{0,1\}^{n\times n}$ and a vector $c$ of size $n$, where the entry $c_i\ge 0$ denotes the cost of directly actuating the  state variable $i$, determine   the sparsity of the  input matrix $\bar B$  that solves the following optimization problem
\begin{align}
\min_{\bar B\in \{0,1\}^{n \times n}} &\qquad\qquad  \|\bar B\|_{c}\label{probP2}\\
\text{s.t.} \qquad & \quad (\bar A,\bar B) \text{ structurally controllable}.\notag
\end{align}

\end{itemize}

Notice that in the above problems, we assume that  the cost associated to each directly actuated/measured state variable can take different values, but is independent of the labeled input/output variable. In addition, observe that in $\mathcal P_1$ and $\mathcal P_2$, some solutions may comprise at most one nonzero entry in each column; in other words, solutions in which each input actuates  at most one state variable. Such inputs are referred to as \emph{dedicated inputs}, and they correspond to columns of the input matrix $\bar{B}$ with exactly one nonzero entry. Additionally, if a   solution $\bar B^*$ is such that all its nonzero columns consist of exactly one nonzero entry, it is referred to as a \emph{dedicated solution}, otherwise it is referred to as  a \emph{non-dedicated solution}. Finally, we notice that a solution to $\mathcal P_1$ or $\mathcal P_2$ may consist of columns with all zero entries, that can be disregarded when considering the deployment of the inputs required to actuate the system.  Notice that in the worst case scenario, taking the identity matrix as the input matrix we obtain structural controllability, which justifies the dimensions chosen for the solution search space.
\vspace{0.4cm}

\begin{remark}
The solution procedures for  $\mathcal{P}_{1}$ and $\mathcal P_2$ also address the corresponding  structural observability output matrix design problem by invoking the duality between observability and controllability in linear time-invariant (LTI) systems~\cite{Hespanha09}.\hfill $\diamond$
\end{remark}

Recently, the I/O selection problem have received increasing attention in the literature, especially, since the publication of~\cite{Olshevsky}. In~\cite{Olshevsky}, the \emph{minimal controllability problem}, i.e., the problem of determining the sparsest input matrix that ensures controllability of a given LTI system, was shown to be NP-hard. Furthermore, in \cite{Olshevsky} some greedy algorithms ensuring controllability are provided, and in~\cite{MinContNPcomp} procedures for obtaining optimal solutions in specific cases are presented. Although the minimal controllability problem is NP-hard, in~\cite{PequitoJournal}, using graph theoretical constructions, the  structural version of the minimal controllability problem, or the \emph{minimal structural controllability problem}, was shown to be polynomially solvable. Note, in turn, this implies that  for almost all numerical realizations of the dynamic matrix, satisfying a predefined pattern of zeros/nonzeros, the associated minimal controllability problem is polynomially solvable. Alternatively, in~\cite{Summers,Tzoumas,Clark,ClarkCDC,Pasqualetti,linfarjovTAC14leaderselection} the configuration of actuators is sought to ensure certain performance criteria; more precisely,  \cite{Summers,Tzoumas,Pasqualetti,ClarkCDC} focus on optimizing properties of the controllability Grammian,  whereas~\cite{Clark} studies leader selection problems in order to minimize convergence errors experienced by the follower agents, given in terms of the norm of the difference between the follower agents' states and the convex hull of the leader agent states, and~\cite{linfarjovTAC14leaderselection} aims to identify leaders that are
most effective in minimizing the deviation from consensus in
the presence of disturbances.  In addition, in~\cite{Clark,Summers,Tzoumas} the submodularity properties of functions of the controllability Grammian are explored, and design algorithms are proposed that achieve feasible placement with certain guarantees on the optimality gap. The I/O selection problem considered in the present paper differs from the aforementioned problems in the following three aspects: first, the selection of the inputs is not restricted to belong to a specific given set of possible inputs, which are referred to as constrained input selection (CIS) problems. Secondly, it contrasts with~\cite{Summers,Tzoumas,Clark,Pasqualetti} in the sense that we do not aim to ensure performance in terms of a function of the controllability Grammian, but we aim to minimize the overall actuation cost, measured in terms of manufacturing/installation/preference costs. Lastly,  instead of optimizing a specific (numerical) system instance, we focus on structural design guarantees that hold for almost all (numerical) system instantiations with a specified dynamic coupling structure. The minimum CIS when only the structure is considered, i.e., the problem of determining the minimum collection of inputs in a CIS problem, to ensure structural controllability,  has previously been extensively studied, see~\cite{CommaultD13,dionSurvey} and references therein. Nevertheless, the minimum CIS problem  is shown to be NP-hard~\cite{NPcompCMIS} in general. In addition, it is not difficult to see that  any problem that aims to minimize the cost of a collection of inputs in a CIS problem, where  arbitrary (non-zero) actuation costs are considered, is at least as difficult as the minimum CIS, hence, also NP-hard. Notice that, from the computational complexity point of view,  this also contrasts with the problems explored in this paper. 

 On the other hand, the problems presented here are  closely related with  those presented in~\cite{PequitoJournal}, where the sparsest (structural) input matrix among all structural input matrices ensuring structural controllability is sought; in particular, the latter corresponds to $\mathcal P_1$ and $\mathcal P_2$ when the costs are uniform, i.e., each variable incurs in the same (non-zero) cost.    In addition, in~\cite{PequitoJournal}  an algorithm   to compute  a minimum subset of state variables that need to be actuated to ensure structural controllability of the system  is described. However, the techniques developed in \cite{Pequito2013,PequitoJournal}, despite providing useful insights, are not sufficient to address the problems $\mathcal{P}_{1}$ and $\mathcal P_2$ with non-uniform cost, i.e., when actuating different state variables incur  in different costs.  More recently, some preliminary results on problems $\mathcal P_1$ and $\mathcal P_2$ were presented by us in \cite{PequitoECC} and  \cite{pequitocdc13} respectively, by exploiting some intrinsic  properties of the class of all minimal subset of state variables that need to be actuated by dedicated inputs to ensure structural controllability. Although \cite{PequitoECC} and  \cite{pequitocdc13} provided polynomial algorithmic solutions to $\mathcal P_1$ and $\mathcal P_2$, they incur in higher polynomial complexity (more precisely, $\mathcal O (n^{3.5})$) than the ones presented in this paper incurring in $\mathcal O(n^3)$; furthermore, the solutions presented in the present paper do not require the exploration of the properties of the feasibility space, and reduces the problem to a well known graph-theoretic problem, the weighted maximum matching problem. Very recently, we came across~\cite{alex14}, where the problem $\mathcal P_1$ is addressed for a specific binary actuation cost structure in which the cost vector is restricted to be of the form  $c\in\{0,\infty\}^{n}$, i.e., corresponding to a subset of state variables that can be actuated with the same finite (zero) cost, and the rest that are \emph{forbidden}, i.e., with an infinite  actuation cost. The proposed method in \cite{alex14} achieves computational complexity $\mathcal O(n^{2.5})$; however, this approach is restricted to the special binary cost structure as described above. Furthermore, it is not likely that algorithms to compute the solution to the problems $\mathcal P_1$ and $\mathcal P_2$ with a lower complexity than the ones presented in the present paper are attainable for general cost vectors $c$ (using the same framework), since they  require the computation of a minimum weighted maximum matching, for which the fastest existing algorithms share a complexity equal to the ones  presented here, see \cite{Duan} for recent advances and literature survey.  Consequently, the  methodologies presented in this paper are more suitable to address $\mathcal P_1$ and $\mathcal P_2$ than those in \cite{PequitoECC} and \cite{pequitocdc13}, when dealing with more general cost structures and higher dimensional LTI systems. In addition, we provide a relation between $\mathcal P_1$ and $\mathcal P_2$ that sheds some light on possible extensions towards less restrictive cost assumptions.

The main contributions of this paper are as follows:  we show that we can solve $\mathcal{P}_1$ and $\mathcal P_2$  resorting to algorithms with complexity $\mathcal O(n^3)$, where $n$ is the dimension of the state space. In addition, these algorithms are obtained  while exploring the relation between these two  problems, i.e.,  using some insights from $\mathcal P_1$   to solve $\mathcal P_2$.

The rest of the paper is organized as follows: Section~\ref{preresults}  reviews results from structural systems and its implications on optimal input-output placement in LTI systems with uniform placement cost; furthermore, it explores some graph theorectical concepts required to obtain the main results of this paper. Section~\ref{mainresults} presents the main results of the paper, in particular, a procedure to determine the minimal cost placement of inputs  in LTI systems, as formulated in $\mathcal{P}_{1}$ and $\mathcal P_2$.  Section~\ref{example}  illustrates the procedures through an example. Finally, Section~\ref{conclusions} concludes the paper, and presents avenues for future research.

\section{Preliminaries and Terminology}\label{preresults}

The following standard terminology and notions from graph theory can be found, for instance in \cite{PequitoJournal}. Let $\mathcal D(\bar A)=(\mathcal X,\mathcal E_{\mathcal X,\mathcal X})$ be the digraph representation of $\bar{A}\in\{0,1\}^{n\times n}$, to be referred to as the  \emph{state digraph}, where the vertex set $\mathcal X$ represents the set of state variables (also referred to as state vertices) and $\mathcal E_{\mathcal X,\mathcal X}=\{(x_i,x_j): \ \bar A_{ji}\neq 0\}$ denotes the set of edges.  Similarly, given $\bar B\in \{0,1\}^{n\times p}$,  we define the  digraph $\mathcal D(\bar A,\bar B)=(\mathcal X\cup \mathcal U,\mathcal E_{\mathcal X,\mathcal X}\cup\mathcal E_{\mathcal U,\mathcal X})$, to be referred to as \emph{system digraph}, where $\mathcal U$ represents the set of input vertices and $\mathcal E_{\mathcal U,\mathcal X}=\{(u_i,x_j):\ \bar B_{ji}\neq 0\}$. Further, by similarity, we have the \emph{state-slack digraph} given by $\mathcal D(\bar A,\bar S)=(\mathcal X\cup\mathcal S,\mathcal E_{\mathcal X,\mathcal X}\cup \mathcal E_{\mathcal S,\mathcal X})$, where $\mathcal S$ represents the set of  slack variables (or vertices).  In addition, given  digraphs  $\mathcal D(\bar A,\bar B)$ and $\mathcal D(\bar A,\bar S)$, we say that they are \emph{isomorphic} to each other, if  there exists a bijective relationship between the vertices and edges of the digraphs that preserves the incidence relation. Finally, since the edges are directed, an edge is said to be an \emph{ outgoing edge} from a vertex $v$ if it starts in $v$, and, similarly, is said to be an \emph{incoming edge}  to $w$ if it ends on $w$. 

%
%

\begin{figure}[!h]
\centering
\includegraphics[scale=0.35]{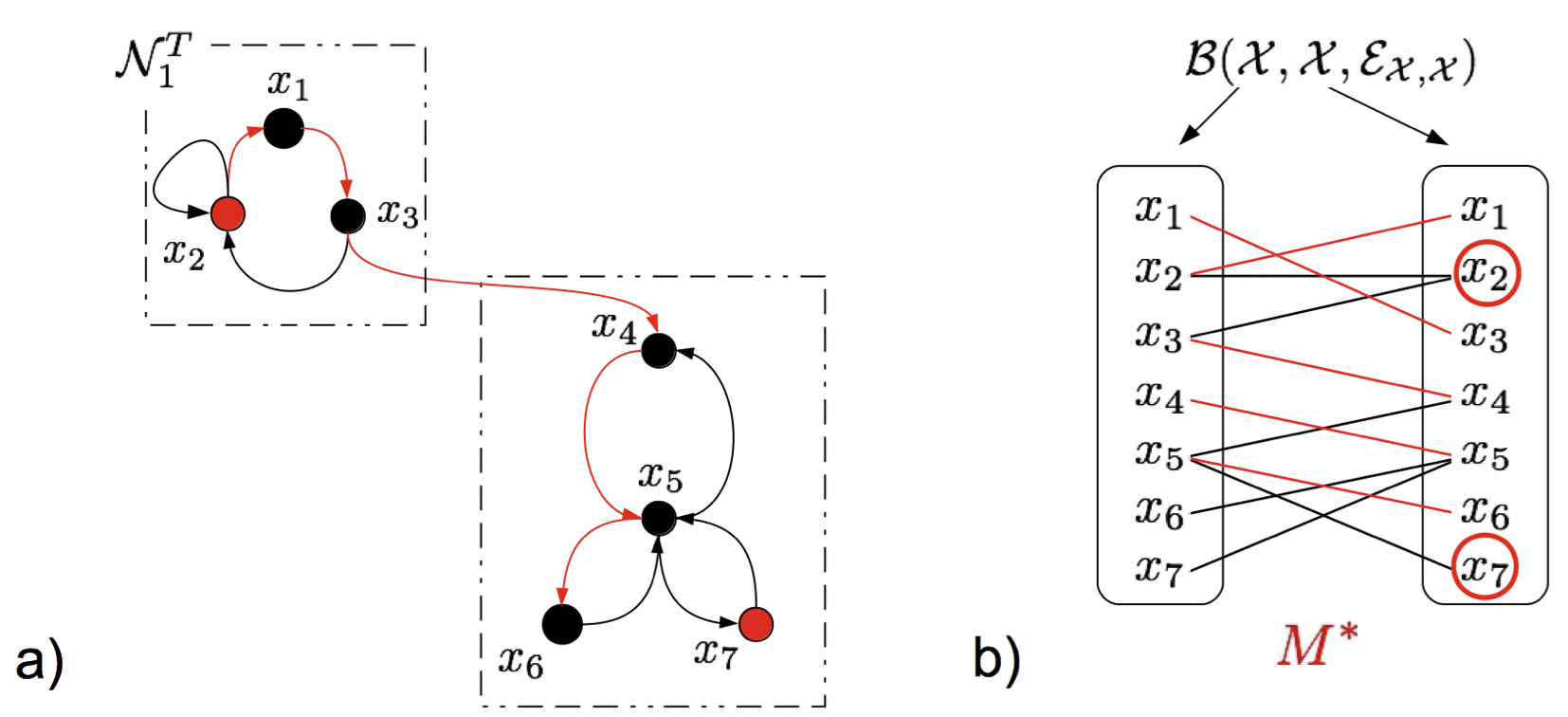}
\caption{ In a) an illustrative example of a  digraph $\mathcal D(\bar A)$ is provided, where the SCCs are depicted in the dashed boxes and the non-top linked SCC (only one) labeled by $\mathcal N^T_1$. The red edges correspond to the edges in a possible maximum matching $M^*$ of the bipartite graph $\mathcal B(\bar A)$, presented in b). Moreover, $x_2,x_7$ are right-unmatched vertices (they do not belong to matching edges) w.r.t. $M^*$ and are depicted by red circles. In fact,  $x_2$ plays \emph{double role}, i.e.,  is a state variable that also belongs to (the only) non-top linked SCC. Finally, by Theorem~\ref{FDIC}, we obtain that  $\{x_2,x_7\}$ is a feasible dedicated input configuration. Hence, the assignment of two distinct inputs to those state variables ensures structural controllability of the system $\mathcal D(\bar A,\bar B)$. Further, it is easy to see that any superset of a feasible dedicated input configuration is also a feasible dedicated input configuration.}
\label{fig:example}
\end{figure}

In addition, we will use the following graph theoretic notions~\cite{Cormen}: A digraph $\mathcal{D}_s=(\mathcal V_s,\mathcal E_s)$ with $\mathcal V_s\subset\mathcal  V$ and $\mathcal E_s\subset \mathcal E$ is called a \textit{subgraph} of $\mathcal{D}=(\mathcal V,\mathcal E)$. A sequence of edges $\{(v_1,v_2),(v_2,v_3),$ $\cdots,(v_{k-1},v_k)\}$, in which all the vertices are distinct, is called \textit{an  elementary path} from $v_1$ to $v_k$. A vertex with an edge to itself (i.e., a \emph{self-loop}), or an elementary path from $v_1$ to $v_k$  together with an additional edge $(v_k,v_1)$, is called a \emph{cycle}.  A digraph $\mathcal{D}$ is said to be strongly connected if there exists a directed path between any two pairs of vertices. A \emph{strongly connected component} (SCC) is a maximal subgraph $\mathcal{D}_s=(\mathcal V_s,\mathcal E_s)$ of $\mathcal{D}$ such that for every $v,w \in \mathcal V_s$ there exists a path from $v$ to $w$. Notice that the SCCs are uniquely defined for a given digraph; consequently, visualizing each SCC as a virtual node (or supernode), we can generate a \textit{directed acyclic graph} (DAG), in which each node corresponds to a single SCC and there exists  a directed edge between two virtual nodes \emph{if and only if} there exists a directed edge connecting vertices within the  corresponding SCCs in the original digraph. The DAG associated with $\mathcal{D}=(\mathcal V,\mathcal E)$ can be efficiently generated in $\mathcal{O}(|\mathcal V|+|\mathcal E|)$~\cite{Cormen}, where  $|\mathcal V|$ and $|\mathcal E|$  denote the number of vertices in $\mathcal V$ and the number of edges in $\mathcal E$, respectively. In the DAG representation, an SCC (a supernode) that has no incoming edge from any state in a different SCC (supernode) is referred to  as  a \emph{non-top linked SCC}, since, by convention, the DAG is graphically represented with edges between the virtual nodes drawn downwards,   an example is depicted in Figure~\ref{fig:example}-a).

For any two vertex sets $\mathcal S_{1}, \mathcal S_{2}\subset \mathcal V$, we define the   \textit{bipartite graph} $\mathcal{B}(\mathcal S_1,\mathcal S_2,\mathcal E_{\mathcal S_1,\mathcal S_2})$, as a  graph (bipartite), whose vertex set is given by $\mathcal S_{1}\cup \mathcal S_{2}$ and the edge set 
$
\mathcal E_{S_1,S_2}\subseteq \{(s_1,s_2)\in \mathcal E \ :\ s_1 \in \mathcal S_1, s_2 \in\mathcal  S_2  \ \}.
$ 
Given $\mathcal{B}(\mathcal S_1,\mathcal S_2,\mathcal E_{\mathcal S_1,\mathcal S_2})$, a matching $M$ corresponds to a subset of edges in $\mathcal E_{\mathcal S_1,\mathcal S_2}$ that do not share vertices, i.e., given edges  $e=(s_1,s_2)$ and $e'=(s_1',s_2')$ with $s_1,s_1' \in \mathcal S_1$ and $s_2,s_2'\in \mathcal S_2$, $e, e' \in M$ only if $s_1\neq s_1'$ and $s_2\neq s_2'$.  A bipartite graph is, by convention,  depicted by a set of vertices $\mathcal S_1$ in the left and the other set of vertices $\mathcal S_2$ in the right to clearly emphasize the bipartition,  an example is depicted in Figure~\ref{fig:example}-b). The vertices in $\mathcal S_1$ and $\mathcal S_2$ are \textit{matched vertices} if they belong to an edge in the  matching $M$, otherwise, we designate the vertices as \textit{unmatched vertices}.  A maximum matching $M^{\ast}$ is a matching $M$ that has the largest number of edges among all possible matchings. It is to be noted  that a maximum matching $M^*$ may not be unique.  For ease of referencing, keeping in mind the bipartite graphical representation,  the term \emph{right-unmatched vertices}, with respect to (w.r.t.) $\mathcal{B}(\mathcal S_1,\mathcal S_2,\mathcal E_{\mathcal S_1,\mathcal S_2})$ and a  matching $M$ (not necessarily maximum), will refer to  those vertices in $\mathcal S_{2}$ that do not belong to a matching edge in $M$, and are denoted by $\mathcal U_R(M)$.  In addition, we introduce the following notation: given a set of edges $\mathcal{E}_{\mathcal S_1,\mathcal S_2}$, we denote by $\mathcal L(\mathcal{E}_{\mathcal S_1,\mathcal S_2})$ and $\mathcal R(\mathcal{E}_{\mathcal S_1,\mathcal S_2})$ the collection of vertices corresponding to the set of left and right endpoints of $\mathcal E_{\mathcal S_1,\mathcal S_2}$, i.e., in $\mathcal S_1$ and $\mathcal S_2$, respectively.

Now, we present some specific bipartite graphs that are closely related with the digraphs previously introduced. More precisely, we have: (i) the \emph{state bipartite graph} $\mathcal B(\bar A)=\mathcal B(\mathcal X,\mathcal X,\mathcal E_{\mathcal X,\mathcal X})$ that we often refer to as the bipartite representation of (or associated with, or induced by) the state digraph $\mathcal D(\bar A)$; (ii) the \emph{system bipartite graph} $\mathcal B(\bar A,\bar B)=\mathcal B(\mathcal X\cup \mathcal U,\mathcal X,\mathcal E_{\mathcal X,\mathcal X}\cup\mathcal E_{\mathcal U,\mathcal X})$ that we often refer to as the bipartite representation of  $\mathcal D(\bar A,\bar B)$; and, similarly to the latter, we have (iii)  the \emph{state-slack bipartite graph} $\mathcal B(\bar A , \bar S)=\mathcal B(\mathcal X\cup \mathcal S,\mathcal X,\mathcal E_{\mathcal X,\mathcal X}\cup\mathcal E_{\mathcal S,\mathcal X})$ that we  often refer to as the bipartite representation of  the state-slack digraph $\mathcal D(\bar A,\bar S)$. 


 If we associate   \emph{weights} (or costs) with  the edges in a digraph and bipartite graph, we obtain a  \emph{weighted digraph} and \emph{weighted bipartite graph}, respectively. A weighted digraph is represented by the pair digraph-weight given by $(\mathcal D=(\mathcal V,\mathcal E);w)$, where $w:\mathcal E\rightarrow \mathbb{R}^+_0\cup\{\infty\}$ is the weight function. Similarly, a weighted bipartite graph is represented by the pair bipartite-weight $(\mathcal B(\mathcal S_1,\mathcal S_2,\mathcal E_{\mathcal S_1,\mathcal S_2});w)$. 
%
%
 Subsequently, we  introduce the  \emph{minimum weight maximum matching} problem. This problem consists in determining the  maximum matching of a weighted bipartite graph $(\mathcal B(\mathcal S_1,\mathcal S_2,\mathcal E_{\mathcal S_1,\mathcal S_2});w)$ that incurs the  minimum weight-sum of its edges; in other words, determining the maximum matching $M^c$ such that
\[
M^c=\arg\min_{M\in \mathcal M} \sum_{e \in M} w(e),
\]
where $\mathcal M$ is the set of all maximum matchings of $\mathcal B(\mathcal S_1,\mathcal S_2,\mathcal E_{\mathcal S_1,\mathcal S_2})$. This problem can be efficiently solved using, for instance,  the Hungarian algorithm~\cite{Munkres1957},  with computational complexity of $\mathcal O( \max\{|\mathcal S_1|,|\mathcal S_2|\}^3)$.  


We will also require the following general results on structural control design from \cite{PequitoJournal} (see also \cite{Pequito2013}). We define a \emph{feasible dedicated input configuration} to be a collection of state variables to which by assigning dedicated inputs  we can  ensure structural controllability of the system. Consequently,  a minimal feasible dedicated input configuration is the minimal subset of state variables to which we need to assign dedicated inputs  to ensure structural controllability. Further, the feasible dedicated input configurations can be characterized as follows.
\vspace{0.4cm}

\begin{theorem}[\textit{Theorem 3 in} \cite{PequitoJournal}]\label{FDIC}
Let $\mathcal D(\bar A)=(\mathcal X,\mathcal E_{\mathcal X,\mathcal X})$ denote the system digraph and $\mathcal B(\bar A) \equiv \mathcal B(\mathcal X,\mathcal X,\mathcal E_{\mathcal X,\mathcal X})$  the associated {state bipartite graph}. Let $\mathcal S_u\subset \mathcal X$, then the following statements are equivalent:
\begin{enumerate}
\item The set  $\mathcal S_u$ is a feasible dedicated input configuration;
\item There exists a subset $\mathcal U_R(M^*)\subset \mathcal S_u$ corresponding to the set of right-unmatched vertices of some maximum matching $M^*$ of $\mathcal B(\bar A) $, and  a subset $\mathcal A_u\subset \mathcal S_u$ comprising one state variable from each non-top linked SCC of $\mathcal D(\bar A)$.\hfill $\diamond$
\end{enumerate}
\end{theorem}

Observe that a state variable can be simultaneously in $\mathcal U_R(M^*)$ and $\mathcal A_u$, even if these sets correspond to those of a minimal feasible dedicated input configuration; thus,   motivating us to refer to those variables as playing a \emph{double role}, since they contribute to both the conditions in Theorem~\ref{FDIC}. To illustrate Theorem \ref{FDIC}, we provide an example in  Figure~\ref{fig:example}.

\vspace{0.4cm}

\begin{remark}\label{GeneralIOsel} In \cite{PequitoJournal,Pequito2013} general results were given on structural input selection, in particular on non-dedicated structural input design, i.e., in which the structural input matrix $\bar{B}$ may possess multiple nonzero entries in each column.  To ease the presentation, we denote by $m$ the number of right-unmatched vertices in any maximum matching of $\mathcal{B}(\bar A)$ and by $\beta$ the number of non-top linked SCCs in $\mathcal{D}(\bar{A})$. The following characterization of structural controllability was obtained in [21,22] (see Theorem 8 in \cite{PequitoJournal}): a pair $(\bar{A},\bar{B})$ is structurally controllable if and only if there exists a maximum matching of $\mathcal{B}(\bar A)$ with a set of right-unmatched vertices $\mathcal{U}_{R}$, such that, $\bar{B}$ has (at least) $m$ nonzero entries, one in each of the  rows corresponding  to the different state variables in $\mathcal{U}_{R}$ and located at different columns, and (at least) $\beta$ nonzero entries, each of which belongs to a row (state variable) corresponding to a distinct non-top linked SCC and located in arbitrary columns.\hfill $\diamond$
\end{remark}

Notice that as a direct consequence of Remark \ref{GeneralIOsel}, we obtain that any $\bar{B}$, such that $(\bar{A},\bar{B})$ is structurally controllable, must have at least $m$ distinct nonzero columns (or $m$ distinct control inputs).

\section{Main Results}\label{mainresults}

Despite  the fact that problems $\mathcal P_1$ and $\mathcal P_2$ seem to be highly combinatorial, hereafter we show that they can  be solved using polynomial complexity (in the dimension of the state space) algorithms. To obtain those results, we first present some  intermediate results:  (i) we characterize the  matchings that the bipartite graphs used in the sequel can have (Lemma~\ref{lemma1} and Lemma~\ref{maxMatAugDigraph}), and (ii) we characterize the minimum weight maximum matchings that a weighted bipartite graph can have, upon a specific cost structure to be used to solve and characterize the solutions to $\mathcal P_1$ and $\mathcal P_2$, see Lemma~\ref{findUnmatched} and Lemma~\ref{reductionMat}.  Subsequently,  we initially   address $\mathcal P_1$ by reducing it to a minimum weight maximum matching problem, and present the solution in Algorithm~\ref{solP1}. Then, motivated by  the solution to $\mathcal P_1$, we provide a reduction of $\mathcal P_2$ to another minimum weight maximum matching problem, described in Algorithm~\ref{solP2}. The correctness and computational complexity proofs of Algorithm~\ref{solP1} and Algorithm~\ref{solP2} are presented in Theorem \ref{theorem1costAutomatica} and Theorem \ref{theorem2costAutomatica}, respectively.

\subsection{Intermediate Results}

Let  $\bar{S}$ be a $n\times q$ structural (binary) matrix, and denote by $\mathcal B(\bar A, \bar S )$ the state-slack bipartite graph  associated with the digraph $\mathcal D(\bar A,\bar S )$. Note, by construction, the state-slack digraph $\mathcal{D}(\bar A, \bar S)$ consists of $n+q$ vertices, where the $q$ additional vertices (in comparison with the state  digraph $\mathcal{D}(\bar A)$) correspond  to the \emph{slack} variables, introduced by $\bar S$.
 Further,  by construction, the slack variables  only have  outgoing edges (associated with the nonzero entries of $\bar S$) to the state  variables in $\mathcal{D}(\bar A,\bar S)$; in other words,  there are no incoming edges into the slack variables.  We start by  relating maximum matchings of the two bipartite graphs $\mathcal{B}(\bar{A})$ and $\mathcal{B}(\bar A,  \bar S)$, where, note that, the former is a subgraph of the latter. This will also help in obtaining better insight and better understanding of  the  properties of the maximum matchings of the different bipartite graphs.

\vspace{0.4cm}

\begin{lemma}
Let $\mathcal B(\bar A, \bar S)=(\mathcal X\cup \mathcal S,\mathcal X,\mathcal E_{\mathcal X,\mathcal X}\cup\mathcal E_{\mathcal S,\mathcal X})$ be the state-slack bipartite graph, $\mathcal B(\bar A)=\mathcal B(\mathcal X,\mathcal X,\mathcal E_{\mathcal X,\mathcal X})$ and $\mathcal B(\bar S)=\mathcal B(\mathcal S,\mathcal X,\mathcal E_{\mathcal S,\mathcal X})$. The following statements hold:
\begin{enumerate}
\item If $M_{\bar A}$ and $M_{\bar S}$ are matchings of $\mathcal B(\bar A)$ and $\mathcal B(\bar S)$ respectively, and $\mathcal R(M_{\bar A})\cap \mathcal R(M_{\bar S})=\emptyset$, then $M_{\bar A,\bar S}=M_{\bar S}\cup M_{\bar A}$ is a matching of  $\mathcal B(\bar A, \bar S)=(\mathcal X\cup \mathcal S,\mathcal X,\mathcal E_{\mathcal X,\mathcal X}\cup\mathcal E_{\mathcal S,\mathcal X})$.
\item If $M_{\bar A,\bar S}$ is a matching of  $\mathcal B(\bar A,\bar S)=(\mathcal X\cup \mathcal S,\mathcal X,\mathcal E_{\mathcal X,\mathcal X}\cup\mathcal E_{\mathcal S,\mathcal X})$, then $M_{\bar A,\bar S}=M_{\bar S}\cup M_{\bar A}$, where $M_{\bar A}=M_{\bar A,\bar S}\cap \mathcal E_{\mathcal X,\mathcal X}$ and $M_{\bar S}=M_{\bar A,\bar S}\cap \mathcal E_{\mathcal S,\mathcal X}$ are (disjoint) matchings of $\mathcal B(\bar A)$ and $\mathcal B(\bar S)$, respectively.
\end{enumerate}
In  particular, $\mathcal R(M_{\bar S})\subset \mathcal U_R(M_{\bar A})$, where $ \mathcal U_R(M_{\bar A})$ is the set of right-unmatched vertices associated with the matching $M_{\bar A}$.

\hfill $\diamond$
\label{lemma1}
\end{lemma}

\begin{proof}
The proof of (1) follows by noticing that, by construction of $\mathcal B(\bar A, \bar S)$, we have $\mathcal L(M_{\bar A})\cap \mathcal L(M_{\bar S})=\emptyset$, and by assumption $\mathcal R(M_{\bar A})\cap \mathcal R(M_{\bar S})=\emptyset$, which implies that  $M_{\bar A,\bar S}=M_{\bar S}\cup M_{\bar A}$ has no edge with common end-points; in other words, it is a matching of $\mathcal B(\bar A, \bar S)=(\mathcal X\cup \mathcal S,\mathcal X,\mathcal E_{\mathcal X,\mathcal X}\cup\mathcal E_{\mathcal S,\mathcal X})$, by definition of matching.

On the other hand, the proof of (2) follows by noticing that the edges in $M_{\bar A,\bar S}$ belong to either $\mathcal E_{\mathcal X,\mathcal X}$ or $\mathcal E_{\mathcal S,\mathcal X}$ and noticing  that $M_{\bar A}$ and $M_{\bar S}$ have no common endpoints since $M_{\bar A,\bar S}$ is a matching. Subsequently, it is easy to see that $M_{\bar A}$ and $M_{\bar S}$ are matchings of $\mathcal B(\bar A)$ and $\mathcal B(\bar S)$, respectively. 
\end{proof}

Subsequently, from Lemma~\ref{lemma1}, we can obtain the following result characterizing the maximum matchings of $\mathcal B(\bar A, \bar S)$.
\vspace{0.4cm}

\begin{lemma}\label{maxMatAugDigraph}
Let $\mathcal B(\bar A, \bar S)=(\mathcal X\cup \mathcal S,\mathcal X,\mathcal E_{\mathcal X,\mathcal X}\cup\mathcal E_{\mathcal S,\mathcal X})$ be the state-slack bipartite graph. If $M^*_{\bar A,\bar S}$ is a maximum matching of  $\mathcal B(\bar A, \bar S)=(\mathcal X\cup \mathcal S,\mathcal X,\mathcal E_{\mathcal X,\mathcal X}\cup\mathcal E_{\mathcal S,\mathcal X})$, then $M^*_{\bar A,\bar S}=M_{\bar S}\cup M_{\bar A}$, where $M_{\bar A}=M_{\bar A,\bar S}\cap \mathcal E_{\mathcal X,\mathcal X}$ and $M_{\bar S}=M_{\bar A,\bar S}\cap \mathcal E_{\mathcal S,\mathcal X}$  are (disjoint) matchings of $\mathcal B(\bar A)$ and $\mathcal B(\bar S)$, respectively, and $M_{\bar S}$ contains the largest collection of edges incoming into a set of right-unmatched vertices of some maximum matching of $\mathcal B(\bar A)$.
In  particular, $\mathcal R(M_{\bar S})\subset \mathcal U_R(M_{\bar A})$, where $ \mathcal U_R(M_{\bar A})$ is the set of right-unmatched vertices associated with the (possibly not maximum) matching $M_{\bar A}$.\hfill $\diamond$
 \end{lemma}
\begin{proof}
From Lemma~\ref{lemma1}-(2),  we obtain that  $M_{\bar A}$ and $M_{\bar S}$ are (disjoint) matchings of $\mathcal B(\bar A)$ and $\mathcal B(\bar S)$, respectively. Now, recall that any set of right-unmatched vertices $\mathcal U_R$ associated with a  matching of a bipartite graph comprises a set of right-unmatched vertices $\mathcal U_R^*$ associated with a maximum matching of that bipartite graph~\cite{PequitoJournal}. Next, given that $M^*_{\bar A,\bar S}$ is a maximum matching of $\mathcal B(\bar A, \bar S)$, it follows that $\mathcal U_R(M^*_{\bar A,\bar S})$ comprises the lowest possible number of right-unmatched vertices. Now, to establish that  $M_{\bar S}$ contains  the largest collection of edges incoming into a set of right-unmatched vertices of a maximum matching of $\mathcal B(\bar A)$, suppose  by contradiction, that this is not the case. Then,   there exists at least one more right-unmatched vertex in the set of right-unmatched vertices associated with a matching $M'$ of  $\mathcal B(\bar A, \bar S)$  than in the set of right-unmatched vertices associated with a maximum matching $M^*_{\bar A,\bar S}$; hence, $M'$ cannot be a maximum matching. 
\end{proof}

\begin{figure}[!h]
\centering
\includegraphics[scale=0.55]{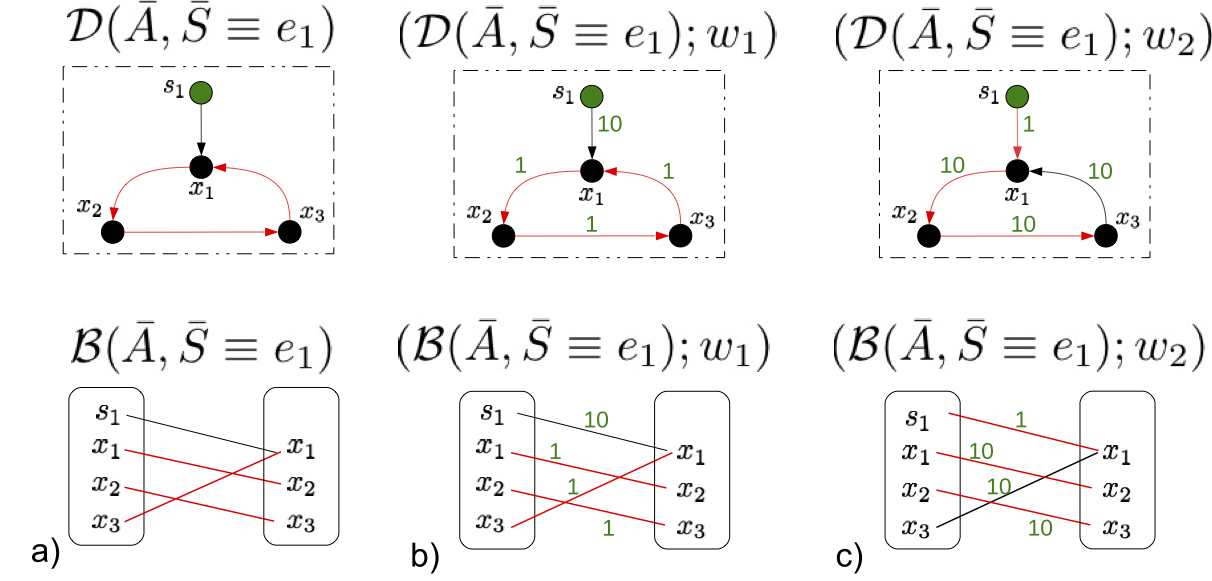}
\caption{  In a) we have a state digraph $\mathcal D(\bar A)$ and its associated state bipartite graph $\mathcal B(\bar A)$, where the edges in the  maximum matching of  the former are depicted in red. By adding a slack variable $s_1$ (corresponding to having the slack matrix equal to the canonical vector $e_1$ with first entry equal to $1$ and zero elsewhere)   depicted by the green vertex, and associating weights with the edges of the extended digraph, we obtain $(\mathcal D(\bar A,\bar S\equiv e_1);w_1)$ and $(\mathcal D(\bar A,\bar S\equiv e_1);w_2)$  depicted in   b) and c), respectively. The new digraph has two possible maximum matchings (the associated bipartite graph), with edges that are depicted in red, shown in b) and c). In addition,  if we consider the minimum weight maximum matching, then b) corresponds to the case where the weight of the edge from the slack variable to $x_1$ is larger than the weight of the edge from $x_3$ to $x_1$; whereas, c) provides the alternative scenario, where edges outgoing from slack variables are preferred to those between state variables. }
\label{fig:MatchingTricks}
\end{figure}

We now extend the results of Lemma~\ref{lemma1} and Lemma~\ref{maxMatAugDigraph} to  weighted bipartite graphs. Consequently, when solving the minimum weight maximum matching, the maximum matching (characterized in  Lemma~\ref{maxMatAugDigraph}) obtained depends on the weights assigned. For example, as depicted in Figure~\ref{fig:MatchingTricks}-c),  if the weights of the edges between state variables in the bipartite graph are larger than the weight of the edge outgoing from the slack variable, then the edge outgoing from the slack variable belongs to the minimum weight maximum matching of the  weighted state-slack  bipartite graph. Alternatively, if the weight of the edge outgoing from the slack variable is larger  than the weights of the edges between state variables, then the minimum weight maximum matching of the weighted state-slack  bipartite graph  comprises the largest set of edges between state variables; in particular, in  Figure~\ref{fig:MatchingTricks}-b) a maximum matching consists of edges between state variables only. The same reasoning can be readily applied if several slack variables are considered, as formally stated next. 

\vspace{0.4cm}

\begin{lemma}\label{findUnmatched} Let  $\bar A\in \{0,1\}^{n \times n}$ and $\bar S\in\{0,1\}^{n\times p}$ with $p\le n$. Consider the weighted state-slack bipartite graph $(\mathcal B(\bar A,\bar S);w)$, where $\mathcal B(\bar A,\bar S)=\big (\mathcal X\cup \mathcal S,\mathcal X,\mathcal E \equiv (\mathcal E_{\mathcal X,\mathcal X}\cup \mathcal E_{ \mathcal S,\mathcal X})\big)$, and $w:\mathcal E\rightarrow \mathbb{R}^+_0\cup\{\infty\}$ such that  $w(e_{\bar S})> w(e_{\bar A})=c_{\bar A}\in\mathbb{R}^+$, with $e_{\bar S}\in \mathcal E_{\mathcal S,\mathcal X}$ and $e_{\bar A}\in \mathcal E_{\mathcal X,\mathcal X}$.
A minimum weighted maximum matching $M^*_{\bar A,\bar S}$ of $(\mathcal B(\bar A, \bar S); w)$ is given by
\[
M^*_{\bar A,\bar S}=M^*_{\bar A}\cup\mathcal E_{\bar S}^*,
\]
where $\mathcal E^*_{\bar S}$ consists in the largest collection of edges incoming into a set of right-unmatched vertices associated with a maximum matching $M^*_{\bar A}$ of $\mathcal B(\bar A)$ and such that  $\mathcal E^*_{\bar S}$  incurs in the  lowest weight-sum among all possible collection of edges incoming into a set of right-unmatched vertices associated with a maximum matching  of $\mathcal B(\bar A)$.
\hfill $\diamond$
\end{lemma}
\begin{proof}
From Lemma~\ref{maxMatAugDigraph} we have that any maximum matching of $\mathcal B(\bar A, \bar S)$ comprises a set $\mathcal E^*_{\bar S}\subset \mathcal E_{\mathcal S,\mathcal X}$ that consists in the largest collection of edges incoming into a set of right-unmatched vertices associated with a maximum matching $M^*_{\bar A}$ of $\mathcal B(\bar A)$.  In addition, since the weights of the edges in $\mathcal E_{\mathcal X,\mathcal X}$ are uniform and   less than the weights of the edges  in $\mathcal E_{\mathcal S,\mathcal X}$, it follows that the edges from $\mathcal E_{\mathcal X,\mathcal X}$ are preferred over the edges in $\mathcal E_{\mathcal S ,\mathcal X}$ as far as the maximum matching $M^*_{\bar A,\bar S}$ is concerned; consequently,  the edges from $\mathcal E_{\mathcal X,\mathcal X}$ in $M^*_{\bar A,\bar S}$ are those that belong to a maximum matching  $M^*_{\bar A}$ of $\mathcal B(\bar A)$. By noticing that the weight-sum of all  matchings of $\mathcal B(\bar A)$ incur in the same cost, and a set $\mathcal E^*_S$ with the characteristics previously described must belong to the maximum matching $M^*_{\bar A,\bar S}$ of $(\mathcal B(\bar A,\bar S);w)$,  the minimum cost of $M^*_{\bar A,\bar S}$ is achieved by considering the set  $\mathcal E^*_{\bar S}$ incurring in the lowest weight-sum,  among all possible collection of edges incoming into a set of right-unmatched vertices associated with a maximum matching  of $\mathcal B(\bar A)$. 
\end{proof}

By reversing the inequality between  the weights of the edges between state variables and those outgoing from the slack variables, we obtain the following result.
\vspace{0.4cm}

\begin{lemma}\label{reductionMat} Let  $\bar A\in \{0,1\}^{n \times n}$ and $\bar S\in\{0,1\}^{n\times p}$ with $p\le n$. Consider the weighted state-slack bipartite graph $(\mathcal B(\bar A,\bar S);w)$, where $\mathcal B(\bar A,\bar S)=\big (\mathcal X\cup \mathcal S,\mathcal X,\mathcal E \equiv (\mathcal E_{\mathcal X,\mathcal X}\cup \mathcal E_{ \mathcal S,\mathcal X})\big)$, and $w:\mathcal E \rightarrow \mathbb{R}^+_0\cup\{\infty\}$ such that   $w(e_{\bar S})< w(e_{\bar A})=c_{\bar A}\in\mathbb{R}^+$, with $e_{\bar S}\in \mathcal E_{\mathcal S,\mathcal X}$ and $e_{\bar A}\in \mathcal E_{\mathcal X,\mathcal X}$.
A minimum weighted maximum matching $M^*_{\bar A,\bar S}$ of $(\mathcal B(\bar A, \bar S); w)$ is given by
\[
M^*_{\bar A,\bar S}=M^*_{\bar S}\cup M_{\bar A},
\]
where $M^*_{\bar S}$ and $M_{\bar A}$ are as given in Lemma~\ref{maxMatAugDigraph}, and  $M^*_{\bar S}$ is a maximum matching of $\mathcal B(\bar S)=\mathcal B(\mathcal S,\mathcal X,\mathcal E_{\mathcal S,\mathcal X})$ whose  edges   incur in the lowest weight-sum among all possible maximum matchings of  $\mathcal B(\bar S)$.
\hfill $\diamond$
\end{lemma}
\begin{proof}
The proof  follows a   similar reasoning to that in the proof of Lemma~\ref{maxMatAugDigraph}. In particular, notice that  $M^*_{\bar S}$ is a maximum matching of $\mathcal B(\bar S)$ because  an arbitrary weight-sum of the edges of a maximum matching of $\mathcal B(\bar S)$ is smaller than that of a collection of edges of $\mathcal B(\bar A,\bar S)$ with the  same size containing edges from $\mathcal E_{\mathcal X,\mathcal X}$, and secondly it consists in the largest collection of edges incoming into a set of right-unmatched vertices associated with a maximum matching $M^*_{\bar A}$ of $\mathcal B(\bar A)$\end{proof}

\subsection{Solution to $\mathcal P_1$}

Now, we present the  reduction of $\mathcal P_1$ to a minimum weight maximum matching problem. Intuitively, given the system's dynamical structure and its digraph representation, we consider an extended digraph with as many slack variables as the minimum number of state variables required to obtain a feasible dedicated input configuration (recall Remark 2).  These slack variables will indicate which state variables should be considered to achieve a feasible dedicated input configuration. Towards this goal, outgoing edges from the slack variables into the state variables (to be considered to the feasible dedicated input configuration)  are judiciously chosen such that a minimum weight maximum matching containing these edges exists; hence, corresponding to the feasible dedicated input configuration that incurs in minimum cost. The systematic reduction of $\mathcal P_1$ to a minimum weight maximum matching,  is presented in Algorithm~\ref{solP1}. Although in Algorithm~\ref{solP1} we determine a solution $\bar B$ that is dedicated, in Remark~\ref{nondedicatedP1}  (using Remark~\ref{GeneralIOsel}), we can characterize all possible solutions to $\mathcal P_1$. Next, we present the proof of correctness of Algorithm~\ref{solP1} and its complexity.

\begin{algorithm}

\textbf{Input:} The structural $n\times n$ system matrix $\bar A$, and the vector  $ c$ of size $n$ comprising the cost of actuating each state variable.

\textbf{Output:} A solution $\bar B$ to $\mathcal P_1$ comprising dedicated inputs.

\textbf{1.} Determine the minimum number  $p$ of dedicated inputs required to ensure structural controllability \cite{PequitoJournal}.

\textbf{2.} Let $\mathcal N^T_j$, with $j=1,\cdots,\beta$, denote the non-top linked SCCs of $\mathcal D(\bar A)$. Let $c_{\max}$ be the maximum real value (i.e., not considering $\infty$) in $c$, and consider  $p$ slack variables, where each slack variable $k=1,\ldots,\beta$ has  outgoing edges to all the state variables in the $k$-th non-top linked SCC $\mathcal N^T_k$, whereas, for  the remaining $p-\beta$ slack variables have outgoing edges  to all state variables, i.e., 
\[
\bar S=\left[\begin{array}{cccc}
| & | & &|\\
\bar s_1 & \bar s_2 & \cdots &\bar s_p \\
| & | & &|
\end{array}\right], 
\]
where the $i$th entry of $\bar s_k$ is given by $[\bar s_k]_i=1$ if $x_i\in\mathcal N^k$ with $k=1,\ldots,\beta$, and $[\bar s_k]_i=0$  otherwise; further, for $k=\beta+1,\ldots,p$ we have $[\bar s_k]_i=1$ for $i=1,\ldots,n$.
Now, consider $(\mathcal B(\bar A,\bar S);w)$ where $w$ is given as follows:

\[
w(e)=\left\{ \begin{array}{cc}
c_{\text{max}}+1, & e\in\mathcal E_{\mathcal X,\mathcal X},\\
c_i, & e\equiv (s_j,x_i) \in\mathcal E_{\mathcal S,\mathcal X}, \ j=1,\ldots,p,\\
\infty, & \text{otherwise}.
\end{array} \right.
\]

%
%
%
\textbf{3.} Determine the minimum weight maximum matching $M^*$ associated with the bipartite graph $(\mathcal B(\bar A,\bar S);w)$.

\textbf{4.}  Assign dedicated inputs to the state variables in $\Theta=\{x\in \mathcal X: \ (s_k, x) \in M^*, \ k=1,\cdots, p\}$. In other words, given  the indices of the state variables in $\Theta$, denoted by $\mathcal J\subset\{1,\cdots,n\}$, set $\bar B=\mathbb{I}_n^{\mathcal J}$, where $\mathbb{I}_n^{\mathcal J}$ is a submatrix of the $n\times n$ identity matrix consisting of the columns with  indices in $\mathcal J$. If $|\mathcal J|=p$ and the weight-sum of $M^*$ is finite,  then $(\bar A,\bar B)$ is structurally controllable,  and a solution to $\mathcal P_1$ is obtained; otherwise,  the problem is infeasible, i.e., there is no feasible $\bar{B}$ (with finite cost) such that $(\bar{A},\bar{B})$ is structurally controllable.

\caption{Solution to $\mathcal P_1$}	
\label{solP1}
\end{algorithm}

\vspace{0.25cm}

\begin{theorem}\label{theorem1costAutomatica}
Algorithm \ref{solP1} is correct, i.e., it provides a solution to $\mathcal P_1$ (as long as the set of feasible $\bar{B}$'s is non-empty). Moreover, its computational complexity is  $\mathcal O(n^3)$. \hfill $\diamond$
\end{theorem}

\begin{proof}
First,  we notice that a solution obtained using Algorithm~\ref{solP1} is feasible, i.e., leads to $\bar B=\mathbb{I}_n^{\mathcal J}$  such that $(\bar A,\bar B)$ is structurally controllable, if $|\mathcal J|=p$ and the weight-sum of $M^*_{\bar A,\bar B}$ is finite.  More precisely,   let $\mathcal D(\bar A)$ be the state digraph comprising  $\beta$ non-top linked SCCs. Further, it is possible to (efficiently) determine a minimal feasible dedicated input configuration~\cite{PequitoJournal}, and we denote its size by $p$. By recalling  Theorem~\ref{FDIC}, we have  at least $\beta$ state variables in different non-top linked SCCs and the remaining $p-\beta$ state variables correspond to right-unmatched vertices in the set of right-unmatched vertices associated with a maximum matching of the state bipartite graph  $\mathcal B(\bar A)$ and do not belong to the non-top linked SCCs. Consequently, we consider an augmented digraph $\mathcal D(\bar A,\bar S)$, with $\bar S\in\{0,1\}^{n\times p}$, with  $p$ slack variables, that will indicate the variables to be considered for obtaining a minimal feasible dedicated input configuration. Further,  we construct $\bar S$ to be as follows: each of the  $\beta$ slack variables are such that slack variable $k$, with $k=1,\ldots,\beta$,   has   only outgoing edges to all the state variables in the non-top linked SCC $k$, and the remaining the remaining $p-\beta$ slack variables have edges to all state variables.   From  Lemma~\ref{maxMatAugDigraph} and  the knowledge that a feasible dedicated input configuration with $p$ state variables exists, we can argue that a maximum matching of $\mathcal B(\bar A, \bar S)$ contains edges outgoing from slack variables and  ending in all right-unmatched vertices  with respect to a maximum matching of $\mathcal B(\bar A)$. Furthermore, there exists  a maximum matching $M^*_{\bar A,\bar S}$ of $\mathcal B(\bar A, \bar S)$,  where all slack variables belong to  matching edges in $M^*_{\bar A,\bar S}$. In the former  case, due to the proposed construction, there  is  at least one edge from a slack variable to each non-top linked SCC; hence, by Theorem~\ref{FDIC}, the collection of the state variables, where the edges  with origin in slack variables belonging to  $M^*_{\bar A,\bar S}$ end, is a feasible dedicated input configuration; such a collection  is also minimal since it has exactly $p$ state variables -- the size of a minimal feasible dedicated input configuration. Therefore, we aim to determine such a matching, which will be accomplished by  considering a minimum weight maximum matching problem. More precisely, we associate weights with the edges of the proposed digraph:   weights of the edges outgoing from the slack variables are set to be equal to the cost of actuating the state variables (specified by the cost vector $c$) corresponding to the end points of the edges, whereas, the remaining edges are assigned large enough (otherwise arbitrary)  weights, in particular, higher than those  of the edges outgoing from the slack variables. Therefore, taking $(\mathcal B(\bar A, \bar S);w)$ to be the weighted version of $\mathcal B(\bar A, \bar S)$ with the weight function as previously described,  by invoking Lemma~\ref{reductionMat}, there exists a  maximum matching  $M^*_{\bar A,\bar S}$ of $\mathcal B(\bar A, \bar S)$, where each edge  with origin in slack variables belonging to  $M^*_{\bar A,\bar S}$  indicates which state  variables should be actuated, and such  collection is a  feasible dedicated input configuration if it is of size $p$ and the sum of the weights of the edges in  $M^*_{\bar A,\bar S}$ is finite. In other words, an infinite cost would correspond to the case where no feasible dedicated input configuration exists, i.e., no finite cost input matrix $\bar{B}$ can make the system structurally controllable.  In summary, we obtain a minimal feasible dedicated input configuration with the lowest cost, which corresponds to a (dedicated) solution to $\mathcal P_1$. In Figure~\ref{fig:mincostFDIC}, we present  an illustrative example, where we present the construction for a digraph consisting of a single (i.e., $\beta =1$) non-top linked SCC and the size of the minimum feasible dedicated configuration is $p=2$.

Now, to conclude that $\bar B$ obtained by Algorithm~\ref{solP1} incurs in the minimum cost, suppose by contradiction that this is not the case. This implies that, there exists another feasible  $\bar B'$ leading to a smaller  cost. If $\bar B'$ has multiple nonzeros in the same column, given  Remark \ref{GeneralIOsel}, there exists $\bar B''$ with the same cost as $\bar B'$ and  with at most one nonzero entry in each column such that $(\bar{A},\bar{B}'')$ is structurally controllable.  Consequently, by letting $\mathcal D(\bar A,\bar B'')=(\mathcal X\cup \mathcal U,\mathcal E_{\mathcal X,\mathcal X}\cup\mathcal E_{\mathcal U,\mathcal X})$ and $\mathcal D(\bar A, \bar S)$ to be isomorphic, and considering the weight function $w$  as in  Algorithm \ref{solP1}, it follows by Lemma~\ref{reductionMat} that there exists a maximum matching $M''$ of $(\mathcal B(\bar A, \bar S)=(\mathcal X\cup \mathcal S,\mathcal E_{\mathcal X,\mathcal X}\cup\mathcal E_{\mathcal S,\mathcal X});w)$ containing $\mathcal E_{\mathcal S,\mathcal X}$. Nevertheless, this is a contradiction since it implies there exists a maximum matching $M''$ incurring in a lower cost than $M^*$ obtained using, for instance, the Hungarian algorithm~\cite{Munkres1957}, and used to construct $\bar B$.

Finally, the computational  complexity follows from noticing that Step 1 has complexity  $\mathcal O(n^3)$ \cite{PequitoJournal}. Step 2 can be computed using linear complexity algorithms. In Step 3, the  Hungarian algorithm is used on the $n\times (n+p)$ matrix obtained at the end of Step 2, and  incurs in $\mathcal O(n^3)$ complexity. Finally, Step~4 consists of a for-loop operation which has linear complexity. Hence, summing up the different complexities, the result follows. 
\end{proof}

\begin{remark}\label{nondedicatedP1}Now, consider $(\mathcal B(\bar A,\bar B)=(\mathcal X\cup \mathcal S,\mathcal E_{\mathcal X,\mathcal X}\cup\mathcal E_{\mathcal U,\mathcal X}),w')$ where $\bar B$  is as obtained from Algorithm~\ref{solP1} and $w'$ is given as follows:

\[
w'(e)=\left\{ \begin{array}{cc}
1, & e\in\mathcal E_{\mathcal X,\mathcal X},\\
2, & e \in\mathcal E_{\mathcal U,\mathcal X},\\
\infty, & \text{otherwise}.
\end{array} \right.
\]
%

Therefore, considering $(\mathcal B(\bar A,\bar B);w)$ and using Lemma~\ref{findUnmatched}, a minimum weight maximum matching comprises the edges from $\mathcal E_{\mathcal U,\mathcal X}$ with end-points in the  state variables that belong to the set of right-unmatched vertices $\mathcal U_R(M^*_{\bar A})$ associated with a maximum matching $M^*_{\bar A}$ of $\mathcal B(\bar A)$. Consequently, from  Remark~\ref{GeneralIOsel} and the dedicated solution obtained with Algorithm~\ref{solP1}, we can further obtain a non-dedicated solution to $\mathcal{P}_{1}$; more precisely, one requires  $m$ distinct inputs, where $m$ is the number of right-unmatched vertices $\mathcal U_R(M^*_{\bar A})$,  assigned to those state variables in $\mathcal U_R(M^*_{\bar A})$ and  some input (potentially the same) must be assigned to the remaining state variables required to ensure structural controllability (identified by the dedicated solution). \hfill $\diamond$
\end{remark}

\subsection{Solution to $\mathcal P_2$}

Next, we present the reduction of $\mathcal P_2$ to a minimum weight maximum matching problem,  similar to the reduction presented in the previous subsection. Nevertheless, because of the (potential) double role of the state variables in a minimal feasible input configuration (i.e., state variables in a non-top linked SCC and right-unmatched vertices), we may be able to further reduce the cost by considering two state variables instead of one playing a double role used in the construction of a minimal feasible dedicated input configuration (associated with a solution to $\mathcal P_1$).  The aforementioned reduction is described in detail in Algorithm \ref{solP2} that provides a dedicated solution. Nevertheless, a general non-dedicated solution to $\mathcal P_2$ is characterized using Remark \ref{GeneralIOsel} and Remark~\ref{nondedicatedP1}. The correctness and computational complexity of Algorithm~\ref{solP2} is presented next.

\begin{algorithm}

\textbf{Input:} The structural $n\times n$  system matrix $\bar A$, and the vector $ c$ of size $n$  comprising the cost of actuating each state variable.

\textbf{Output:} A solution $\bar B$ to $\mathcal P_2$ comprising dedicated inputs.

\textbf{1.} Determine the minimum number  $p$ of dedicated inputs required to ensure structural controllability \cite{PequitoJournal}.

\textbf{2.} Let $\mathcal N^T_j$, with $j=1,\cdots,\beta$, denote the non-top linked SCCs of $\mathcal D(\bar A)$. Additionally, let $c_{\max}$ be the maximum real value (i.e., not considering $\infty$) in $c$, and $c_{\min}^k$ corresponds to the minimum cost associated with the state variables in $\mathcal N^T_k$.  In addition, consider  $p$ slack variables, where each slack variable has  outgoing edges to all the state variables.

Now, consider $(\mathcal B(\bar A,\bar S);\tilde w)$ where $\tilde w$ is given as follows:

\[
\tilde w(e)=\left\{ \begin{array}{ccl}
c_{\text{max}}+1, & e\in\mathcal E_{\mathcal X,\mathcal X},&\\
c_i, & e\equiv (s_k,x_i) \in\mathcal E_{\mathcal S,\mathcal X} & \text{ and } x_i\in\mathcal N^k,\\
&& \ k=1,\ldots,\beta,\\
c_i+c_{\text{min}}^k, & e\equiv (s_k,x_i) \in\mathcal E_{\mathcal S,\mathcal X} &\text{ and } x_i\notin\mathcal N^k, \\ & & \ k=1,\ldots,\beta,\\
c_i, & e\equiv (s_k,x_i) \in\mathcal E_{\mathcal S,\mathcal X},& k=\beta+1,\ldots,p,\\
\infty, & \text{otherwise}.
\end{array} \right.
\]

%
%
%

\textbf{3.} Determine the  minimum weight maximum matching $M^*$ associated with the bipartite graph $(\mathcal B(\bar A, \bar S);\tilde w)$.

\textbf{4.} Let $M^*=\{(s_k,x_{\sigma(k)}) : k=1,\cdots, p\}$ where $\sigma(.)$ is a permutation of the state variables indices.  Assign dedicated inputs to the state variables in $\Theta'$, given  as follows 
\[
\Theta'= \bigcup_{k=1,\dots, p} \Omega_k, 
\]
where
\[
 \Omega_k=\left\{\begin{array}{ll}
\{x_{\sigma(k)}\}, \text{ if } x_{\sigma(k)}\in \mathcal N^T_k,\\
\{x_{\sigma(k)},x_{min}^k\}, \text{ if } x_{\sigma(k)}\notin \mathcal N^T_k
\end{array}
\right.
\]
with $x_{min}^k$ a state variable in $\mathcal N^T_k$ with the minimum cost. In other words, given  the indices of the state variables in $\Theta'$, denoted by $\mathcal J'\subset\{1,\cdots,n\}$, set $\bar B=\mathbb{I}_n^{\mathcal J'}$, where $\mathbb{I}_n^{\mathcal J'}$ is a submatrix of the $n\times n$ identity matrix consisting of the columns with  indices in $\mathcal J'$. If the weight-sum of $M^*$ is finite, then $(\bar A,\bar B)$ is structurally controllable;  and a solution to $\mathcal P_2$ is obtained; otherwise, the problem is infeasible, i.e., there is no feasible $\bar{B}$ (with finite cost) such that $(\bar{A},\bar{B})$ is structurally controllable.

\caption{Solution to $\mathcal P_2$}	
\label{solP2}
\end{algorithm}
\vspace{0.4cm}

\begin{theorem}\label{theorem2costAutomatica}
Algorithm \ref{solP2} is correct, i.e., it provides a solution to $\mathcal P_2$ (as long as the set of feasible $\bar{B}$'s is non-empty). Moreover, its computational complexity is of  $\mathcal O(n^3)$. \hfill $\diamond$
\end{theorem}

\begin{proof}
To establish the feasibility of the solution, we first notice that no more than two state variables are required to substitute one  state variable with double role  to ensure structural controllability, since additional state variables increase the cost unnecessarily. The reduction proposed is similar to that which reduces $\mathcal P_1$ to a minimum weighted maximum matching, but we need to consider new edges from the slack variables with outgoing edges only to the state variables in a single SCC (and used to ensure the existence of a state variable in  each non-top linked SCC) to the remaining state variables.  The weight of these new edges consists of the cost of a state variable in the corresponding non-top linked SCC  and incurring in the  lowest cost, plus the cost of the state  variable where the edge ends on.  Therefore, if such an edge is selected in the minimum weight maximum matching of $(\mathcal B(\bar A,  \bar S);w)$,  it means that it is cheaper to actuate two state variables than  actuating one state variable with double role. In summary, the edges selected in the minimum weight maximum matching contain sufficient information about the state variable(s) to which dedicated inputs need to be assigned to to establish structural controllability, as in Theorem \ref{FDIC}. Furthermore, this choice of state  variables minimizes the cost, which implies a solution to $\mathcal P_2$. 

To prove that  the input matrix $\bar B$ obtained in Algorithm~2 is a solution to $\mathcal P_2$, suppose by contradiction that there exist another feasible $\bar B'$ that incurs in lower cost. By similar arguments as used in the proof of Theorem \ref{theorem1costAutomatica}, there exists a $n\times p$ matrix $\bar B''$ with at most one non-zero entry in each column and incurring in the same cost as $\bar B'$. Consequently,  $p''$ (dedicated) inputs are assigned to either right-unmatched vertices associated with a maximum matching of the state bipartite graph or state variables in a non-top linked SCC. Let us denote by $p''_r$ the number of right-unmatched vertices and $p''_n$ the number of state variables in the non-top linked SCCs that are not right-unmatched vertices, where it is easy to see that $p''_r+p''_n=p''$. Now, consider the construction of $\tilde w$ in Algorithm \ref{solP2}, we have that the $p''$ edges selected can correspond to either one of the following: i) an edge with  cost $c_i$ which corresponds to an edge ending in one of the $p''_r$ state variables; or ii) an edge with the cost $c_i+c^k_{\min}$ which corresponds to an edge ending in one of the $p''_r$ state variables and another  edge ending in one of the $p''_n$ state variables. Nevertheless, these costs have to sum to the lowest cost, which corresponds to the solution of a minimum weight maximum matching matching determined, for instance, using the Hungarian algorithm~\cite{Munkres1957}. Consequently, we obtain a contradiction since $\bar B$ was constructed using a minimum weight maximum matching with the same weight function, which implies that it has to incur in the same cost.

The complexity can be obtained in a similar fashion as in the proof of Theorem~\ref{theorem1costAutomatica}.
\end{proof}

\begin{remark}\label{nondedicatedP2}
Reproducing the steps  explained in Remark \ref{nondedicatedP1},  non-dedicated input matrices that are solution to $\mathcal P_2$ can be determined from the  dedicated solutions to $\mathcal P_2$  computed by Algorithm~\ref{solP2}.
\vspace{-0.8cm}

\hfill $\diamond$
\end{remark}

\section{Illustrative Example}\label{example}

Consider the examples in Figure \ref{fig:mincostFDIC}  and Figure~\ref{fig:mincost}, in which the manipulating costs for each state variable are  given by: 
\[
\begin{array}{cccccccc}
c_I=&[ 50 & \infty &  10 & 10 & 1 &  10 & 20 ].
\end{array}
\]

The solutions to $\mathcal P_1$ and $\mathcal P_2$ are now presented,  resorting to Algorithm~\ref{solP1} and Algorithm~\ref{solP2}. Steps 2-3 from Algorithm~\ref{solP1}  and Algorithm~\ref{solP2} are  depicted in Figure~\ref{fig:mincostFDIC} and Figure~\ref{fig:mincost}, respectively.

We start by describing the solution to $\mathcal P_1$. As illustrated in Figure~\ref{fig:example}, the minimum number of dedicated inputs required to ensure structural controllability is $p=2$. Consequently, by Algorithm~\ref{solP1}, two slack variables,  denoted by $s_1,s_2$, are introduced. From each slack variable,  new edges to the state variables are introduced to obtain the bipartite graph $(\mathcal B(\bar A, \bar S);w)$ as described in Algorithm~\ref{solP1}; from $s_1$ only edges to the state variables in $\mathcal N^T_1$ are introduced with weights equal to the cost of actuating the respective variables, although only the edges with finite weight are depicted in Figure \ref{fig:mincostFDIC}, whereas from $s_2$  edges to all state variables are introduced. The solution of Step~3 in Algorithm \ref{solP1} is  given by $M^*$, and depicted  by the red edges in Figure  \ref{fig:mincostFDIC}-b). The red vertices correspond to the right-unmatched vertices associated with the maximum matching of the state bipartite graph,  to which dedicated inputs are assigned in Step 4 of Algorithm \ref{solP1}. Notice that $x_1$ is a variable with double role, i.e., besides being a right-unmatched vertex, it is also a variable in a non-top linked SCC.  The total cost of this assignment is equal to the sum of the costs  of actuating $x_1$ and $x_6$, summing up to $60$.

We now show that, by considering a scenario where more variables can be manipulated, we can achieve a lower cost for the presented example. Similarly to the previous example, we have two slack variables, from which edges are created to each of the state variables. In contrast to the previous example, the edges from $s_1$ also include the variables not in $\mathcal N^T_1$ (depicted by blue edges in Figure~\ref{fig:mincost}) with weights equal to the cost of actuating the state variable they end in plus the minimum cost of actuating a   state variable in $\mathcal N^T_1$, i.e.,  $x_3$, with cost  $10$. The minimum weight maximum matching  of $(\mathcal B(\bar A,\bar S);w')$ is given by $M'$ and depicted by red edges in Figure \ref{fig:mincost}-b). Again, the red vertices represent the right-unmatched vertices with respect to a maximum matching of the state bipartite graph, to which dedicated inputs are assigned in Step~4 of Algorithm~\ref{solP2}. Additionally, since one of the edges in blue belongs to $M'$, in this case departing from $s_1$, a dedicated input must be assigned to a state variable from $\mathcal N^T_1$ with the minimum cost, i.e., $x_3$. More precisely, a possible solutions consists in assigning dedicated inputs to the state variables $x_3,x_4,x_6$, incurring in a total cost of $30$.

\begin{figure}[!h]
\centering
\includegraphics[scale=0.35]{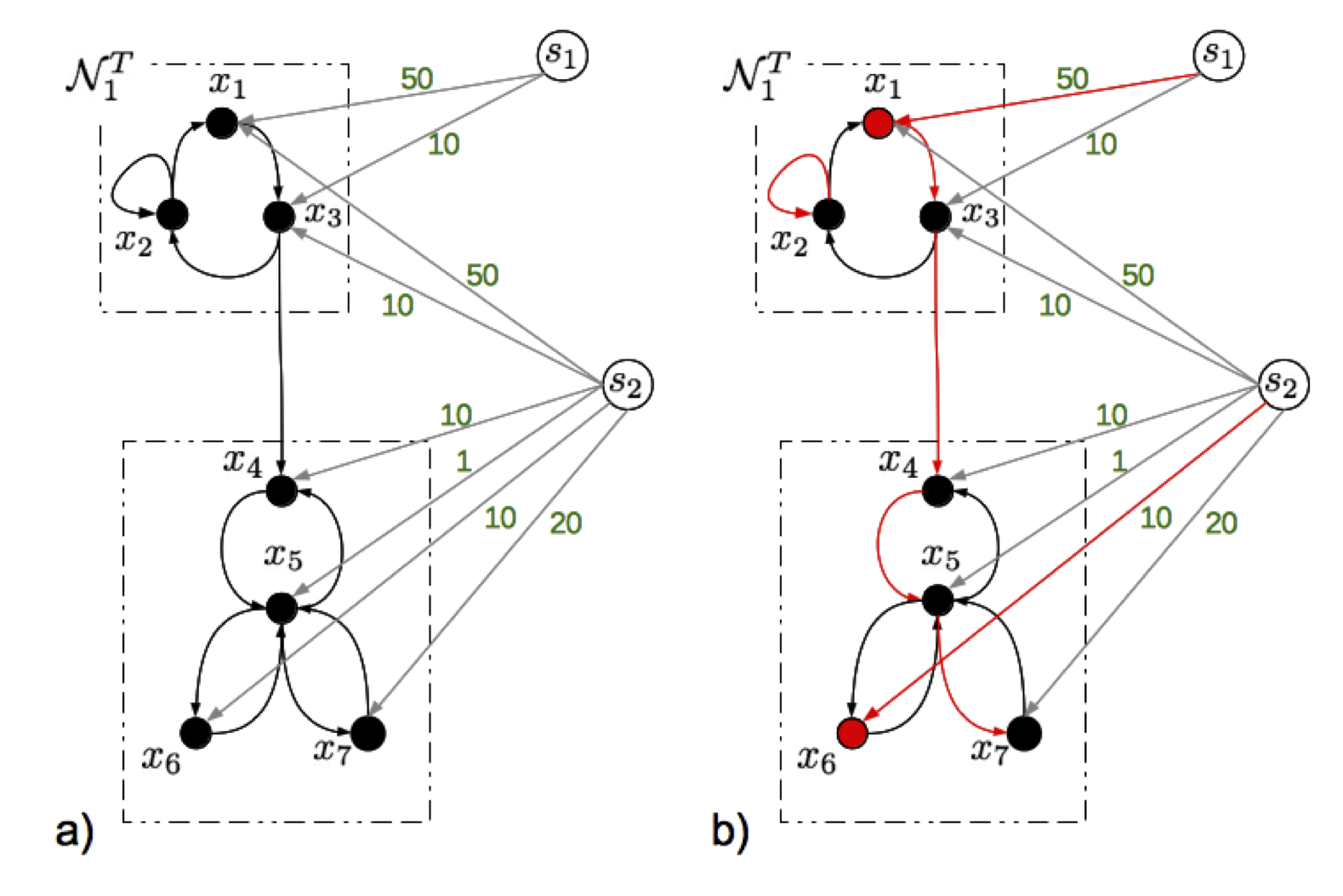}
\caption{ An illustrative example of a  digraph $\mathcal D(\bar A)$, where the SCCs are inscribed in the dashed boxes and the non-top linked SCC labeled by $\mathcal N^T_1$. The edges' costs are depicted by green labels,  the edges that are not depicted have  $\infty$ cost, and the edges between state variables have cost equal to $51$. The red edges correspond to the edges in a possible maximum matching $M$ of the bipartite graph $(\mathcal B(\bar A,\bar S);w)$. The red vertices represent the right-unmatched vertices with respect to a matching $M$ associated with $\mathcal B(\bar A)$. In a) the reduction presented in Step 2 of Algorithm \ref{solP1} is illustrated, where as many slack variables as the minimum number of dedicated inputs required to ensure structural controllability  are considered;  and b) the red edges represent the edges associated with a maximum matching associated with $(\mathcal B(\bar A,\bar S);w)$.}
\label{fig:mincostFDIC}
\end{figure}

\begin{figure}[!h]
\centering
\includegraphics[scale=0.35]{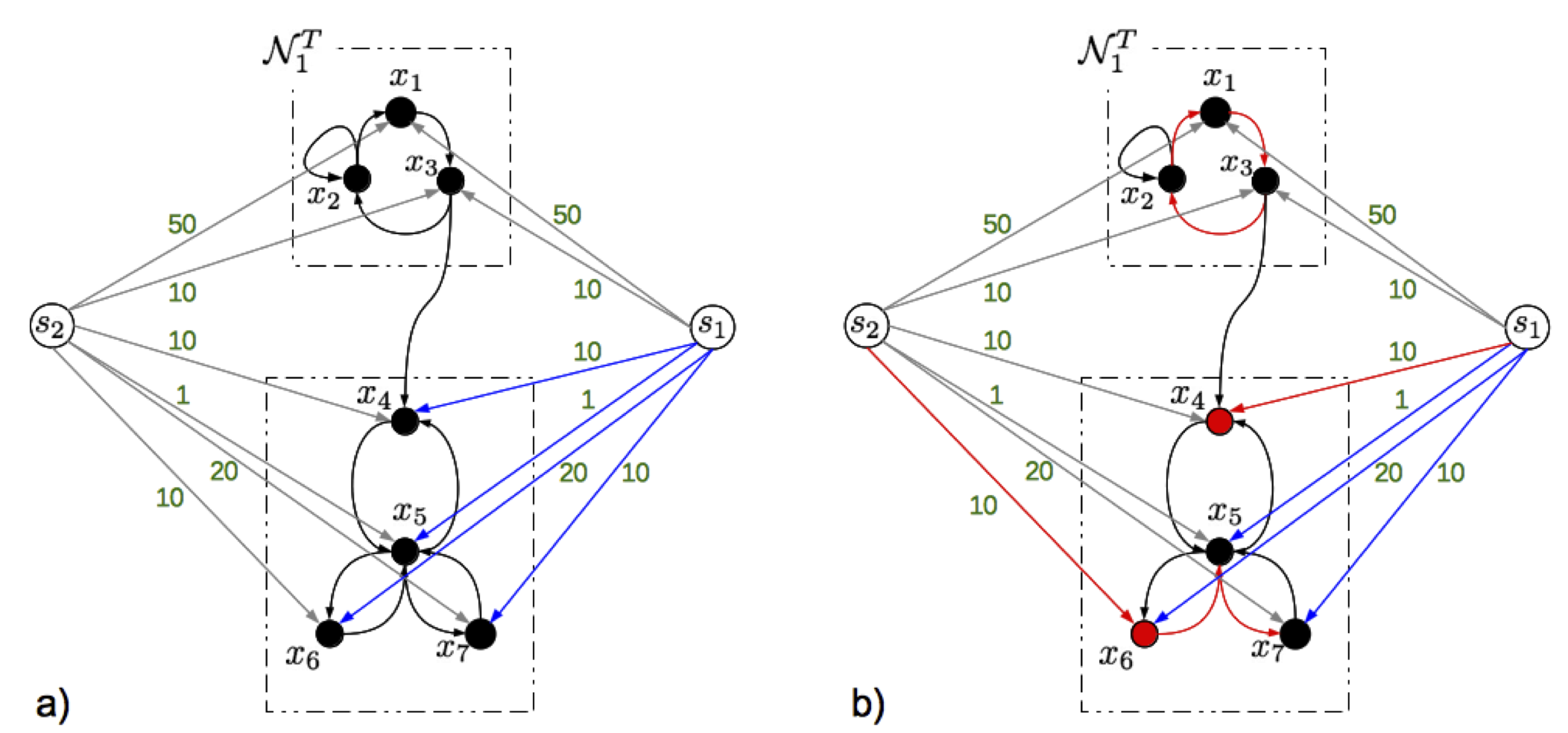}
\caption{ An illustrative example of a  digraph $\mathcal D(\bar A)$, where the SCCs are comprised in the dashed boxes and the non-top linked SCC labeled by $\mathcal N^T_1$. The edges' costs are depicted by green labels,  the edges that are not depicted have  $\infty$ cost, and the edges between state variables have cost equal to $51$. The red edges correspond to the edges in a possible maximum matching $M$ of the bipartite graph $(\mathcal B(\bar A,\bar S);w')$. The red vertices represent the right-unmatched vertices with respect to a matching $M$ associated with $\mathcal B(\bar A)$. Additionally, the edges depicted in blue correspond to the ones with weights consisting in the sum of the minimum cost associated with a state variable in a non-top linked SCC and the cost of the variable in which the edge ends.  In a) the reduction  presented in Step 2 of Algorithm~\ref{solP2} is illustrated, where as many slack variables as the minimum number of dedicated inputs required to ensure structural controllability are considered;  and b) the red edges represent the edges  in a maximum matching associated with $(\mathcal B(\bar A, \bar S);w')$. }
\label{fig:mincost}
\end{figure}

\section{Conclusions and Further Research}\label{conclusions}

In this paper, we provided a systematic method with polynomial   complexity (in the number of the state variables) in order to obtain  minimal cost placements of actuators  ensuring structural controllability of a given LTI system. The proposed solutions  holds under arbitrary non-homogeneous positive assignment costs  for the manipulation of the state variables. By duality, the results extend to the corresponding structural observability output design under cost constraints. The non-homogeneity of the allocation cost makes the framework particularly applicable to input (output) topology design in large-scale dynamic infrastructures, such as power systems, which consist of a large number of heterogeneous dynamic components with varying overheads for controller (sensor) placement and operation.  To the best of our knowledge, relaxing any of these constraints would lead to a strictly combinatorial problem for which no polynomial algorithms are expected to exist. Consequently,  future research may focus on the development of approximation algorithms under relaxed cost assumptions. Towards this goal, some of the  techniques and insights presented in this paper may be useful.

\small
\bibliographystyle{plain}        
\bibliography{automatica13}           

\end{document}